\documentclass[10pt]{article} 

\usepackage{subfig}
\usepackage{multirow} 
\usepackage{tikz}
\usetikzlibrary{positioning}
\usepackage{amsmath}
\usepackage{amsthm}
\usepackage{amsfonts}
\usepackage{amssymb}
\usepackage{array}
\usepackage[utf8]{inputenc}
\usepackage{geometry} 
\usepackage{graphicx} 
\usepackage{algorithm}
\usepackage{algorithmic}
\usepackage[T1]{fontenc}
\usepackage[utf8]{inputenc}
\usepackage{authblk}

\newtheorem{thm}{Theorem}[section]

\newtheorem{cor}[thm]{Corollary}
\newtheorem{lem}[thm]{Lemma}

\newtheorem{defn}{Definition}

\newcommand{\ds}{\displaystyle}
\newcommand{\la}{\lambda}

\newcommand{\C}{\mathcal{C}}
\newcommand{\F}{\mathcal{F}}

\newcommand{\LL}{\mathcal{L}}

\newcommand{\SSS}{\mathcal{S}}

\title{n-Dimensional Optical Orthogonal Codes, Bounds and Optimal Constructions }

\author{T.~L.~Alderson  \thanks {The author acknowledges the support of the NSERC of Canada Discovery Grant program.}}

\affil[1]{University of New Brunswick, Canada}

\begin{document}
\maketitle



\begin{abstract}
We generalized to higher dimensions the notions of optical orthogonal codes. We establish uper bounds on the capacity of general $ n $-dimensional OOCs, and on specific types of ideal codes (codes with zero off-peak autocorrelation).   The bounds are based on the Johnson bound, and subsume many of the bounds that are typically applied  to codes of dimension three or less.  We also present two new constructions of ideal codes; one furnishes an infinite family of optimal codes for each dimension $ n\ge 2 $, and another which provides an asymptotically optimal family for each dimension $ n\ge 2 $. The constructions presented are based on certain point-sets in finite projective spaces of dimension $k$ over $GF(q)$ denoted $PG(k,q)$.  
\end{abstract}

%



\section{Introduction}

A (1-dimensional) $(n,w,\lambda_a,\lambda_c)$ optical orthogonal code (OOC) is a family of binary sequences
(codewords) of length $n$, and constant Hamming weight $w$ satisfying the following two conditions:
\begin{itemize}
\item (off-peak auto-correlation property) for any codeword
$c=(c_0,c_1,\ldots,c_{n-1})$ and for any integer $1\leq t \leq n-1$, we have $\ds
\sum_{i=1}^{n-1}c_ic_{i+t}\leq \lambda_a$,
\item (cross-correlation property) for any two distinct codewords
$c,c'$ and for any integer $0\leq t \leq n-1$, we have $\ds \sum_{i=0}^{n-1}c_ic'_{i+t}\leq
\lambda_c$,
\end{itemize}
where each subscript is reduced modulo $n$.

An $(N,w,\lambda_a,\lambda_c)$ OOC with $\lambda_a=\lambda_c$ is denoted an $(N,w,\lambda)$ OOC.
The number of codewords is the \emph{size} of the code.  For fixed values of $N$, $w$, $\lambda_a$
and $\lambda_c$, the largest size of an $(N,w,\lambda_a,\lambda_c)$-OOC is denoted
$\Phi(N,w,\la_a,\la_c)$. An $(N,w,\la_a,\la_c)$-OOC of size $\Phi(N,w,\la_a,\la_c)$ is said to be
\textit{optimal}.

\begin{def} \label{asym} A  family of $ (N,w,\lambda_a,\lambda_c)  $ OOCs   is called asymptotically optimal if 
\begin{equation}
\lim_{N\to \infty} \frac{|C|}{\Phi(C)}=1.
\end{equation}
\end{def}

Since the work of Salehi \textit{et. al.} \cite{Salehi1989} \cite{Salehi1989a},   OOCs have been employed within optical code division multiple access (OCDMA) networks.  OCDMA  networks are widely employed due to their strong performance with multiple users. They are ideally suited for bursty, asynchronous, concurrent traffic.  In applications, optimal OOCs facilitate the largest possible number of asynchronous users to transmit information efficiently and reliably. %
In order to maintain low correlation values the code length must increase quite rapidly with the number of users, reducing bandwidth utilization.  

The $ 1 $-D OOCs spread the input data bits only in the time domain. Technologies such as wavelength-division-multiplexing (WDM) and dense-WDM enable the spreading of codewords  in both space and time \cite{Park1992}, or in wave-length and time \cite{Mendez1995a}.  Hence, codewords may be considered as $\Lambda \times T$ $(0,1)$-matrices. These codes are referred to in the literature as multiwavelength, multiple-wavelength, wavelength-time hopping, and 2-dimensional OOCs (2-D OOCs).

This addition of another dimension allows codes with off-peak autocorrelation zero and  thereby improving the OCDMA performance in comparison with 1-D OCDMA. For optimal constructions of 2-D OOC's see \cite{Alderson20111187,1315909,1523307}, and for asymptotically optimal constructions see \cite{OmraniEK04,OmraniEK06,MR1641021,486611,649764}.  
Later, a third dimension was added  which gave an increase the code size and the performance of the code \cite{doi:10.1117/12.238940,Alderson2017}. In 3-D OCDMA the optical pulses are spread in three domains space, wave-length, and time, with codes referred to as \textit{space/wavelength/time spreading} codes, or \textit{3-D OOC}. In  \cite{Alvarado2015}, coherent   fibre-optic   communication   systems are discussed, whereby   both quadratures and both polarizations of the electromagnetic field are  used, resulting in a  four-dimensional signal  space.

In the present work we carry these developments to the next natural stage, introducing constructions and bounds on $ n $-dimensional OOCs, for all $ n\ge 1 $. In section \ref{sec:nD OOCs and bounds} we introduce $ n $-dimensional OOCs.  We develop some upper bounds on these codes based on the Johnson Bound. We also develop bounds on higher dimensional ideal codes ($ \la_a=0 $). In Section \ref{sec: New Constructions} we present two new constructions of ideal codes; one infinite family of optimal codes, and another which is asymptotically optimal.     

\subsection{$ n $-D OOCs and Bounds} \label{sec:nD OOCs and bounds}


Denote by $( \Lambda_1\times \Lambda_2  \cdots \times \Lambda_{n-1} \times T,w, \lambda_a, \lambda_c)$  an $ n $ dimensional Optical Orthogonal Code  ($ n$-D OOC)  with constant weight $w$, $i $'th spreading length $ \Lambda_i $, $ 1\le i \le {n-1} $, and time-spreading length $T$. Each codeword may be considered as  an $ n $-dimensional $\Lambda_1\times \Lambda_2  \cdots \times \Lambda_{n-1} \times T$ binary array.  The off-peak autocorrelation, and cross correlation of
 an $(\Lambda_1\times \Lambda_2  \cdots \times \Lambda_{n-1} \times T,w, \lambda_a, \lambda_c)$ $ n $-D OOC have the following properties.
 
 \begin{itemize}
 	\item (off-peak auto-correlation property) for any codeword
 	$A=(a_{i_1,i_2,\ldots,i_n})$ and for any integer $1\leq t \leq T-1$, we have\\ $\ds
 	\sum_{i_1=0}^{\Lambda_1-1}\sum_{i_2=0}^{\Lambda_2-1} \cdots \sum_{i_{n-1}=0}^{\Lambda_{n-1}-1}\sum_{i_n=1}^{T-1}a_{i_1,i_2,\ldots,i_n}a_{i_1,i_2,\ldots,i_n+t}\leq \lambda_a$,
 	\item (cross-correlation property) for any two distinct codewords
 	$A=(a_{i_1,i_2,\ldots,i_n})$, $B=(b_{i_1,i_2,\ldots,i_n})$ and for any integer $0 \leq t \leq T-1$, we have\\ $\ds
 	 	\sum_{i_1=0}^{\Lambda_1-1}\sum_{i_2=0}^{\Lambda_2-1} \cdots \sum_{i_{n-1}=0}^{\Lambda_{n-1}-1}\sum_{i_n=1}^{T-1}a_{i_1,i_2,\ldots,i_n}b_{i_1,i_2,\ldots,i_n+t}\leq \lambda_c$,
 \end{itemize}
 where each subscript is reduced modulo $T$. In the case that $ \la_a=\la_c $, $ C $ is denoted an $( \Lambda_1\times \Lambda_2  \cdots \times \Lambda_{n-1} \times T,w, \lambda)$ OOC.

We note that taking all but $ t-1 $ of the $ \Lambda_i $'s to be $ 1 $ results in a $ t $-dimensional OOC. As with other OOCs we shall take minimal correlation values  to be most desirable.  Codes satisfying $ \la_a=0 $ will be said to be \textit{ideal}. 

 As it is of interest to construct codes with as large cardinality as possible, we now discuss some upper bounds on the size of codes.
 We shall require the following notation. By an $ (N,w,\lambda)_{m+1} $-code, we denote a code of  length $ N $, with constant weight $ w $,  and maximum Hamming correlation (the number of non-zero agreements between the two codewords) of $ \lambda $ over an alphabet (containing zero) of size $ m+1 $. For binary codes ($ m=1 $) the subscript $ 2 $ is typically dropped. Let $ A(N,w,\lambda)_{m+1}$ denote the maximum size of  an $ (N,w,\lambda)_{m+1} $-code. In \cite{Alderson2017}, the following bound is established.

 \begin{thm}[\cite{Alderson2017},Johnson Bound Non-binary] \label{thm:JB}
 	\[   A(N,w,\lambda)_{m+1} \le \left\lfloor
 	\frac{m N}{w} \left\lfloor \frac{m(N-1)}{w-1}\left\lfloor 
 	\cdots \left\lfloor %
 	\frac{ m(N-\lambda)}{w-\la}\right\rfloor \right\rfloor \cdots \right\rfloor\right.
 	.\]
 If $ w^2>mN\lambda $ then
 \[
 A(N,w,\lambda)_{m+1}\le \text{ min } \left\{mN,\left\lfloor \frac{mN(w-\lambda)}{w^2-mN\lambda} \right \rfloor \right\}.
 \]
 	
 \end{thm}
 
From the Johnson Bound for constant weight codes it follows \cite{MR1022081} that
\begin{align} \label{JB}
 \Phi(N,w,\la)\leq J(N,w,\la)& =\left\lfloor
\frac1w \left\lfloor \frac{N-1}{w-1}\left\lfloor \frac{N-2}{w-2} \left\lfloor \cdots \left\lfloor
\frac{N-\lambda}{w-\la}\right\rfloor \right\rfloor \cdots \right\rfloor\right.\right.\\[2mm]
& := \left\lfloor f(N,w,\la)\right\rfloor .
\end{align}

We note that the first bound in Theorem \ref{thm:JB} may also be found in \cite{OmraniEK04}.

Observe that by choosing a fixed linear ordering of the array entries, each codeword from a $ (\Lambda_1\times \Lambda_2  \cdots \times \Lambda_{n-1} \times T,w,\la)$ $ n $-D OOC $ C $  can be viewed as a binary constant weight ($ w $) code of length $ N=\Lambda_1 \Lambda_2  \cdots  \Lambda_{n-1}  T $. Moreover, by including the $ T $ distinct cyclic shifts of each codeword we obtain a corresponding constant weight binary code of size $ T \cdot |C| $.  It follows that
\begin{equation}\label{eqn: JB1}
|C|\le \left \lfloor \frac{A(N, w, \la)}{T}\right \rfloor
\end{equation} 

From  the equation (\ref{eqn: JB1}) and Theorem \ref{thm:JB} we obtain the following bounds for n-D OOCs. 

\begin{thm}[Johnson Bound for $ n $-D OOCs]\label{thm: ddjb2} 
	If $ C $ is an $ (\Lambda_1\times \Lambda_2  \cdots \times \Lambda_{n-1} \times T, w, \la) $ OOC, then
	\begin{align} \label{eqn:3djba} 
	\Phi(C) & \le J(\Lambda_1\times \Lambda_2  \cdots \times \Lambda_{n-1} \times T, w, \la)\\[2mm] & =  \left\lfloor
	\frac{N}{Tw} \left\lfloor \frac{N-1}{w-1}\left\lfloor 
	\cdots \left\lfloor
	\frac{ N -\lambda}{w-\la}\right\rfloor \right\rfloor \cdots \right\rfloor\right.
	\\[2mm] 
     & := \left\lfloor f(\Lambda_1\times \Lambda_2  \cdots \times \Lambda_{n-1} \times T, w, \la)\right\rfloor, 
	\end{align}
where $ N=\Lambda_1 \Lambda_2  \cdots  \Lambda_{n-1}T $. If $ w^2> N \lambda $ then
\begin{equation} \label{eqn:3djbb}
\Phi(C) \le \text{ min } \left\{\frac{N}{T},\left\lfloor \frac{\frac{N}{T}(w-\lambda)}{w^2- N\lambda} \right \rfloor \right\}.
\end{equation}	
\end{thm}	

We note that the  bounds in Theorem \ref{thm: ddjb2} subsume the Johnson type bounds on 1, 2, and 3-dimensional codes, such as those found in    \cite{Alderson20111187, MR1022081,Ortiz-Ubarri2011}. 
Moreover, we can see from the theorem, that in a certain sense, maximum capacity is more intrinsically linked to the time spreading length than to the other dimensions. 

\begin{cor}\label{cor: onlyTmatters}
If $N=\Lambda_1\cdot\Lambda_2\cdots\Lambda_{s-1} \cdot T=\Lambda'_1\cdot\Lambda'_2\cdots\Lambda'_{t-1} \cdot T$ where $ s , t \ge 1 $ then 
\begin{equation}\label{eqn: Johnson inequalities}
    J(\Lambda_1\times\cdots\times\Lambda_{s-1} \times T,w,\la)=  J(\Lambda'_1\times\cdots\times\Lambda'_{t-1} \times T,w,\la)
   \end{equation}
\end{cor}

Some easy arithmetic gives the following.

\begin{lem}\label{lem: bound relationship}
If $N=\Lambda_1\cdot\Lambda_2\cdots\Lambda_{n-1} \cdot T$, then 
\begin{align}\label{eqn: Johnson inequalities}
    \frac{N}{T}\cdot J(N,w,\la) & \le J(\Lambda_1\times\cdots\times\Lambda_{n-1} \times T,w,\la)\\
    &\le \frac{N}{T}\cdot J(N,w,\la)+\frac{N}{T}-1
\end{align}
In particular, if $f(N,w,\la)-J(N,w,\la)<\frac{T}{N}$  (such as the case in which $f(N,w,\la)$ is integral) then
    \begin{equation}
    \frac{N}{T}\cdot J(N,w,\la)= J(\Lambda_1\times\cdots\times\Lambda_{n-1} \times  T,w,\la).
    \end{equation}
\end{lem}


\begin{cor}\label{cor: bound relationship_n>1}
Let $N=\Lambda_1\cdot\Lambda_2\cdots\Lambda_{n-1} \cdot T$ where $T=\Lambda_n\cdot T'$.  If $f(\Lambda_1\times\cdots\times\Lambda_{n-1} \times T ,w,\la)-J(\Lambda_1\times\cdots\times\Lambda_{n-1} \times T ,w,\la)<\frac{1}{\Lambda_n}$, then
    \begin{align}
    \Lambda_n\cdot J(\Lambda_1\times\cdots\times\Lambda_{n-1} \times T ,w,\la)& = J(\Lambda_1\times\cdots\times\Lambda_{n-1}\Lambda_n \times T' ,w,\la)\\
    & =J(\Lambda_1\times\cdots\times\Lambda_{n-1}\times \Lambda_n \times T' ,w,\la).
    \end{align}
\end{cor}

As observed in \cite{Alderson2017} for 3-dimensional OOCs,  an $ n $-D OOC $ C $ with $ \la_a=0 $ can be viewed as a constant weight ($ w $) code of length $ \frac{N}{T} =\Lambda_1 \Lambda_2  \cdots  \Lambda_{n-1} $ over an alphabet of size $ T+1$ containing zero. By including the $ T $ distinct cyclic shifts of each codeword we obtain a corresponding constant weight code of size $ T \cdot |C| $.

It follows that
\begin{equation}\label{eqn: JB}
|C|\le \left \lfloor \frac{A(\frac{N}{T}, w, \la)_{T+1}}{T}\right \rfloor .
\end{equation} 

From Theorem \ref{thm:JB} and the equation (\ref{eqn: JB}) we obtain the following bound for ideal $  $n-D OOCs.

\begin{thm}\label{thm:ideal3djb}[Johnson Bound for Ideal n-D OOC]\\
	Let $ C $ be an $ (\Lambda_1\times\cdot\times\Lambda_{n-1} \times T,w,0,\la_c) $ OOC, then 
\begin{align} 
\Phi(C) & \le J(Ideal) 
\nonumber \\[1ex]
& \label{eqn: joptimal ideal}= \left\lfloor
\frac{N}{Tw} \left\lfloor \frac{ N-T}{w-1}%
\left\lfloor \frac{N-2T}{w-2} %
\left\lfloor \cdots \left\lfloor \frac{N-\lambda T}{w-\la_c}\right\rfloor \right\rfloor \cdots \right\rfloor\right.\right.
\end{align}      
where  $N=\Lambda_1\cdot\Lambda_2\cdots\Lambda_{n-1} \cdot T$. In particular, if $ C $ has (maximal) weight  $w =  \frac{N}{T} $, then $ \Phi(C)\le  T^\lambda $. 	
\end{thm}

Note that the bound (\ref{eqn: joptimal ideal}) is tight in certain cases, see \textit{e.g.} the codes constructed in \cite{1523307}.

\subsection{Ideal Codes and Sections}

Suppose $ A $ is a codeword from an  $ n $-dimensional  $ (\Lambda_1\times\Lambda_2\times \cdot\times\Lambda_{n-1} \times T,w,\la_a,\la_c) $ OOC.  For any fixed $ i $, $ 1\le i \le n-1 $, a $ \Lambda_i $ plane of $ A $ may be considered as an $ (n-1) $-dimensional array. Such a plane is called a $ \Lambda_i $ \textit{section}, or an \textit{ $ i $-section} of $ A $.
  
\begin{figure}[H]
	\centering
	\subfloat[]
	{\begin{tikzpicture}[on grid, scale=.7]
	\draw[yslant=-0.5,]  (0,0) rectangle +(3,3);
	\draw[yslant=-0.5] (0,0) grid (3,3);
	
	\foreach \x in {1,2,3}{
		\node  at (\x -0.6, -\x/2 -0.2) {$ s_{\x} $};
		\node at (-0.4,3-\x +0.5 ) {$ \lambda_{\x} $};
		\node  at (\x -0.7,  \x/2 +3.1) {$ t_{\x} $};		
	}
	
	\foreach \x in {2}{
		\foreach \y in {1}{
			\fill[yslant=-0.5, gray] (\x,\y) rectangle +(1,1);
	}}
	\foreach \x in {0}{
		\foreach \y in {2}{
			\fill[yslant=-0.5, gray] (\x,\y) rectangle +(1,1);
	}}
	\draw[yslant=0.5] (3,-3) rectangle +(3,3);
	\draw[yslant=0.5] (3,-3) grid (6,0);
	\foreach \x in {5}{
		\foreach \y in {-1}{
			\fill[yslant=0.5, gray] (\x,\y) rectangle +(1,1);
	}}
	\foreach \x in {3}{
		\foreach \y in {-2}{
			\fill[yslant=0.5, gray] (\x,\y) rectangle +(1,1);
	}}
	\draw[yslant=0.5,xslant=-1,] (6,3) rectangle +(-3,-3);
	\draw[yslant=0.5,xslant=-1] (3,0) grid (6,3);
	\foreach \x in {5}{
		\foreach \y in {0}{
			\fill[yslant=0.5,,xslant=-1, gray] (\x,\y) rectangle +(1,1);
	}}
	\foreach \x in {3}{
		\foreach \y in {2}{
			\fill[yslant=0.5,,xslant=-1, gray] (\x,\y) rectangle +(1,1);
	}} 
	\end{tikzpicture} \label{Fig:3d word}}
	\subfloat[]\;\;\;
{\begin{tikzpicture}[on grid, scale=.7]
	\draw[yslant=-0.5,]  (0,0) rectangle +(1,3);
	\draw[yslant=-0.5] (0,0) grid (1,1);
	
	\node  at (1 -0.6, -1/2 -0.2) {$ s_{1} $};
	
	\foreach \x in {0}{
		\foreach \y in {2}{
			\fill[yslant=-0.5, gray] (\x,\y) rectangle +(1,1);
	}}
	
	\draw[yslant=0.5] (1,-1) rectangle +(3,3);
	\draw[yslant=0.5] (1,-1) grid (4,2);
	
	\foreach \x in {1}{
		\foreach \y in {1}{
			\fill[yslant=0.5, gray] (\x,\y) rectangle +(1,1);
	}}
	
	\draw[yslant=0.5,xslant=-1,] (3,2) rectangle +(3,1);
	\draw[yslant=0.5,xslant=-1] (3,2) grid (6,3);
	
	\foreach \x in {3}{
		\foreach \y in {2}{
			\fill[yslant=0.5,,xslant=-1, gray] (\x,\y) rectangle +(1,1);
	}}
	\end{tikzpicture}\label{Fig:2sec}}\;\;\;
	\subfloat[]{\begin{tikzpicture}[on grid, scale=.7]
	
	\foreach \x in {0}{
		\foreach \y in {2}{
			\fill[yslant=-0.5, gray] (\x,\y) rectangle +(1,1);
	}}
	
	\foreach \x in {5}{
		\foreach \y in {-1}{
			\fill[yslant=0.5, gray] (\x,\y) rectangle +(1,1);
	}}

	\foreach \x in {5}{
		\foreach \y in {0}{
			\fill[yslant=0.5,,xslant=-1, gray] (\x,\y) rectangle +(1,1);
	}}
	\foreach \x in {3}{
		\foreach \y in {2}{
			\fill[yslant=0.5,,xslant=-1, gray] (\x,\y) rectangle +(1,1);
	}} 
	
	\node  at (-0.6, 2.5) {$ \lambda_{1} $};
		
	\draw[yslant=-0.5] (0,2) grid (3,3);
	
	\draw[yslant=0.5] (3,-1) rectangle +(3,1);
	
	\draw[yslant=0.5] (3,-1) grid (6,0);
		
	\draw[yslant=0.5,xslant=-1] (3,0) grid (6,3);
	\end{tikzpicture}\label{Fig:1sec} }
\caption{ Two sections of a 3-D, $ \Lambda\times S\times T $ ($ =\Lambda_1\times \Lambda_2 \times T $) codeword.  Figure (b) is a  2-section, whereas  (a) is a 1-section.}  \label{Fig:2}
\end{figure}
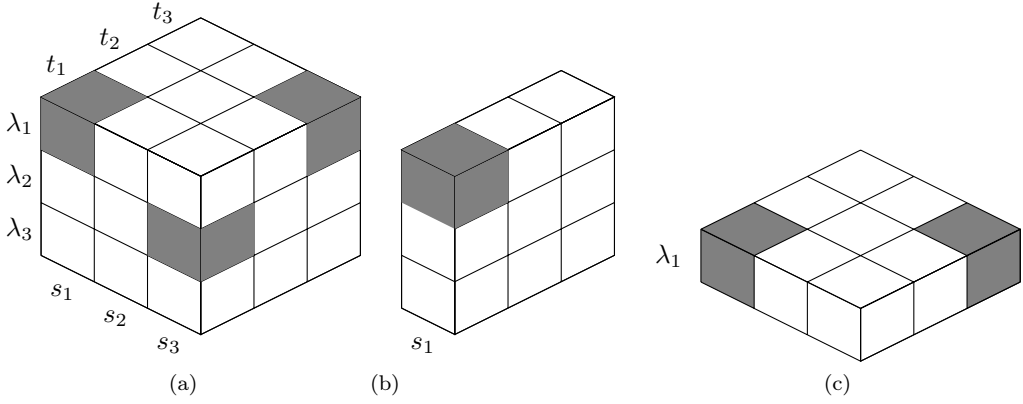
For $ i\ne j $, the intersection of an $ i $-section and a $ j $-section is a section of degree $ 2 $, denoted an $ (i,j) $-section.  
 A section of degree $ t\ge 3 $ is defined in the  analogous way, denoted an $ (i_1,i_1,\ldots,i_t) $-section.

One way to ensure an $ n $-D OOC is ideal, is to restrict the code to having \textit{at most one pulse per $ i $-section}, for some fixed $ i $. Such a code is said to be  AMOPS$(i) $.  For 2-D OOCs these are the At Most One Pulse Per Wavelength (AMOPW) codes \cite{1523307,1315909,Alderson20111187}. For 3-D codes these are At-Most-One-Pulse-per-Plane (AMOPP) codes \cite{MR3320355,AldersonWSEAS2017}.\\
If an $ n $-D OOC, $ C $, is restricted to having at most one pulse per $ (i_1,i_2,\ldots,i_j) $-section, where $ 1\le j \le n-1 $, then $ C $ will be  ideal, and is said to have \textit{at most one pulse per section of degree $ j $}, and is denoted an AMOPS$(i_1,i_2,\ldots,i_j)$  code. If such a code has exactly one pulse per $(i_1,i_2,\ldots,i_j) $-section, then it is said to have a \textit{single pulse per section of degree $ j $}, and is denoted an SPS$(i_1,i_2,\ldots,i_j)$ code.   An ideal $ n $-dimensional OOC is necessarily AMOPS$ (1,2,\ldots,n-1) $.  It is readily seen  that an   AMOPS$(i_1,i_2,\ldots,i_j) $ corresponds to a
constant weight 1-dimensional code of length $ m=\Lambda_{i_1}\cdot\Lambda_{i_2} \cdots \Lambda_{i_j} $ over an alphabet of size  $ \frac{N}{m}+1 $ (containing zero). Consequently, we obtain the following
bounds on AMOPS codes.

\begin{thm}\label{thm:amopsjb}\label{thm:amopsjb}[Johnson Bound for AMOPS codes]

Let $ C $ be an  (ideal) $ (\Lambda_1\times\cdot\times\Lambda_{n-1} \times T,w,0,\la) $-AMOPS${(i_1,i_2,\ldots,i_j)}$  OOC, where $ j\ge 1 $ then 
\begin{align} 
\Phi(C)  
& \le J(AMOPS) %
\nonumber \\[1ex]
& \label{eqn:JBamopsj} = \left\lfloor
\frac{N}{Tw} \left\lfloor \frac{ N\left(1-\frac{1}{M}\right)}{w-1}%
\left\lfloor \frac{ N\left(1-\frac{2}{M}\right)}{w-2} %
\left\lfloor \cdots \left\lfloor \frac{ N\left(1-\frac{\lambda}{M}\right)}{w-\la}\right\rfloor \right\rfloor \cdots \right\rfloor\right.\right.\\[1ex] 
&  \nonumber \le J(Ideal) %
\end{align}  
where $N=\Lambda_1\cdot\Lambda_2\cdots\Lambda_{n-1} \cdot T$, and $ M=\Lambda_{i_1}\cdot\Lambda_{i_2}\cdots\Lambda_{i_{j}} $.    
In the extremal case where $ w = M $, the bound (\ref{eqn:JBamopsj}) simplifies to $\displaystyle   \frac{N^{\lambda+1}}{TM^{\lambda+1}} $.\\ 
In particular, if  $ C $ is an $ (\Lambda_1\times\cdot\times\Lambda_{n-1} \times T,w,0,\la) $-AMOPS$(i)$ OOC, then 
\begin{align} 
\Phi(C) & \le  \label{eqn: JB AMOPSi} \left\lfloor
\frac{N}{Tw} \left\lfloor \frac{ N\left(1-\frac{1}{\Lambda_i}\right)}{w-1}%
\left\lfloor \frac{ N\left(1-\frac{2}{\Lambda_i}\right)}{w-2} %
\left\lfloor \cdots \left\lfloor \frac{ N\left(1-\frac{\lambda}{\Lambda_i}\right)}{w-\la}\right\rfloor \right\rfloor \cdots \right\rfloor\right.\right.
\end{align}  
where $N=\Lambda_1\cdot\Lambda_2\cdots\Lambda_{n-1} \cdot T$. 
In the extremal case where $ w =  \Lambda_i  $, the bound (\ref{eqn: JB AMOPSi}) simplifies to $\displaystyle  T^\lambda \prod_{j\ne i}\Lambda_j^{\lambda+1} $.      
\end{thm}	

The bound  (\ref{eqn:JBamopsj}) is tight in certain cases, see \textit{e.g.} the codes constructed in \cite{ Alderson2017, Alderson20111187,Kim2000, 1315909, Li2012,MR3320355}. We also note that the bound  (\ref{eqn:JBamopsj})  reduces to the bound in Theorem \ref{thm:ideal3djb} when $ j=n-1. $


%

\section{Iterative Constructions of Optimal n-D OOCs}\label{sec: 1D give 2D}

Suppose $ C $ is a $ (\Lambda\times T, w , \la_a,\la_c) $ $ 2 $-D OOC where $ \Lambda=\Lambda_1\cdot\Lambda_2 $. Each codeword in $ C $ can be considered as a $ \Lambda\times T $ array. Let $ X\in C $ where $ X=(x_{i,j}) $. We may construct a corresponding 3-D $ \Lambda_1\times \Lambda_2\times T $ codeword $ Y_x=(y_{i,j,k}) $, $ 0\le i <\Lambda_1,0\le j<\Lambda_2, 0\le k<T $, where    
 \begin{equation}
 y_{i,j,k} = c_{i+j\Lambda_1,k}.
 \end{equation}

It is readily verified that $ C'=\{Y_x\mid x\in  C\} $ is a $ (\Lambda_1\times \Lambda_2\times T, w, \la_a,\la_c) $ $ 3 $-D OOC with $ |C'|=|C|$. Inductively we arrive at the following.

\begin{lem}\label{lem: onlyTmatters_capacity}
Let $ \Lambda=\Lambda_1\cdot\Lambda_2\cdots\Lambda_{s-1} $. There exists an $ (\Lambda\times T, w, \la_a,\la_c) $ 2-D OOC with capacity $ M $ if and only if there exists an $ (\Lambda_1\times\Lambda_2\cdots \times \Lambda_{s-1} \times T, w, \la_a,\la_c) $ $ s $-D OOC with capacity $ M $. 
\end{lem}

An $ n $-D OOC meeting any of the Johnson-type bounds established in the previous sections is referred to as a \textit{J-optimal} code. With reference to Lemma \ref{lem: onlyTmatters_capacity} along with Corollary \ref{cor: onlyTmatters} we observe that a higher dimensional OOC with time spreading length $ T $ obtained from a  J-optimal lower dimensional OOC  by factoring the $ \Lambda_i $'s  will always be J-optimal. For example, each of the optimal codes in Table \ref{table:1} give rise to optimal codes of dimension $ 4 $ or more.

 \begin{table}[H]
 \label{table:1}
 \centering
 	\caption{ \label{table:1}J-optimal ideal $ (\Lambda_1\times\Lambda_2\times T) $ 3D OOCs. Unless stated otherwise, $ \la_c=1 $.}
 	$p$ a prime, $q $ a prime power, $ \theta(k,q)=\frac{q^{k+1}-1}{q-1} $
 	\[
 	\begin{array}{|p{7cm}|l|c|c|}
 	\hline
 	\textrm{ Conditions } & \textrm{ Type } & \textrm{ Capacity }   & \textrm{ Reference}\\
 	\hline 
 	$ w=\Lambda_1\le p $ for all $ p $ dividing $ \Lambda_2 T $  &    SPS(1)   &  \Lambda^2T   & \cite{Kim2000} \\
 	\hline
 	$ w=q+1=\Lambda_1,$ $  \Lambda_2 = q>3 $, $ T=p>q $&  SPS(1) &  \Lambda^2T    &   \cite{Li2012}  \\
 	\hline
 	$w=4=\Lambda_1\le\Lambda_2 = q $, $ T\ge 2 $&  SPS(1) &  \Lambda^2T     &   \cite{Li2012}  \\
 	\hline
 	$ w=3=\Lambda_1 $, 	$ \Lambda_2 $ and $T $ have the same parity  & SPS(1) &  \Lambda^2T    & \cite{MR3320355}\\
 	\hline
 	$ w=3 $,  $ \Lambda T(S-1) $ even,  $ \Lambda T(S-1)S\equiv 0 $   mod  $  3$, and   $ S\equiv 0,1  $ mod $ 4 $ if $ T\equiv 2 $ mod $ 4 $  and $ \Lambda $ is odd.   &   \multirow{2}{*}{AMOPS(1)} &    \multirow{2}{*}{ $ \ds \frac{\Lambda^2 T (S^2-S) }{ 6 } $ }  &  \multirow{2}{*}{\cite{MR3320355}}\\
 	\hline
 	&&&\\[-1.1em]  
 		$ w=\Lambda_1\Lambda_2\le p $ for all $ p $ dividing $ T $  &    \textrm{Ideal}    &  \Lambda^2T   & \cite{Kim2000} \\
 		\hline
 	$ w=q, $	$  \Lambda S T= q^k-1, T= q-1 $ & \textrm{Ideal} &  \left\lfloor
 		\frac{\Lambda S}{q} \left\lfloor \frac{ T(\Lambda S-1)}{q-1}%
 \right\rfloor\right.  &   \cite{Alderson2017} \\ [1em]
 	\hline
 	$ w=q^2, $	$  \Lambda S = q^2+1, T= q+1, $ $\la_c=q-1 $  & \textrm{Ideal}  &  \Lambda S   &   \cite{Alderson2017} \\
 	\hline
 	\end{array}
 	\]
 \end{table}

On the other hand, a  J-optimal $ n $-D OOC may correspond to an $ s $-D  OOC with $ s<n $ that is strictly asymptotically optimal. For example, from the bound (\ref{eqn:JBamopsj}), we see that a J-optimal $ (5\times 5\times 5,5,0,1) $-AMOPS(1) OOC has capacity $ 125 $, whereas a  J-optimal $ (25 \times 5,5,0,1) $-AMOPS(1) OOC has capacity $ 150 $.

\begin{cor}
Let $ \Lambda=\Lambda_1\cdot\Lambda_2\cdots\Lambda_{n-1} $ be a positive integral factorisation.
\begin{enumerate}
\item If there exists a (resp. asymptotically) J-optimal $ (\Lambda\times T, w, \la_a,\la_c) $ 2-D OOC then there exists a (resp. asymptotically) J-optimal $ (\Lambda_1\times\Lambda_2\cdots \times \Lambda_{n-1} \times T, w, \la_a,\la_c) $ $ n $-D OOC.

\item If there exists a (resp. asymptotically) J-optimal $ (\Lambda_1\times\Lambda_2\cdots \times \Lambda_{n-1} \times T, w, \la_a,\la_c) $ $ n $-D OOC, then there exists a  $ (\Lambda\times T, w, \la_a,\la_c) $ 2-D OOC  which is at least asymptotically J-optimal. 
\end{enumerate} 
\end{cor}


\begin{thm}\label{thm: nD-nD}
Let $C$ be an $(\Lambda_1\times\Lambda_2\times \cdots \times\Lambda_{n-1} \times T,w,\la_a,\la_c)$ $ n $-D OOC, $ n\ge 1 $. For any positive integral factorization $ T=T_1\cdot T_2 $, there exists an $(T_1\Lambda_1\times\Lambda_2\times \cdots \times \Lambda_{n-1} \times T,w,\la_a',\la_c')$ $ n $-D OOC, $C'$ with  $\la_a'\le\la_a $, $ \la_c'\le max\{\la_a,\la_c\}$, and $|C'|= T_1\cdot |C|$.
\end{thm}
\begin{proof}
For $ n=1,2 $ see Theorems 3 and 5 in \cite{MR2553388}. The result then follows from Lemma \ref{lem: onlyTmatters_capacity}. 
\end{proof}

There are many constructions of optimal 1-dimensional OOCs. From the Theorem \ref{thm: nD-nD} we see that in some cases optimal 1-dimensional OOCs give optimal $ n $-D OOCs.  

\begin{cor}\label{cor: 1D optimal to nD optimal}
Let $C$ be an $(N,w,\la)$ OOC with $N=\Lambda_1\cdot \Lambda_2\cdots \Lambda_{n-1}\cdot T$.
\begin{enumerate}
\item   If $C$ is J-optimal and $f(N,w,\la)-J(N,w,\la)<\frac{T}{N}$, then a J-optimal $((\Lambda_1\times \Lambda_2\times\cdots \times \Lambda_{n-1}\times T, w, \la)$ $ n $-D OOC exists.
\item   If $C$ is a member of a  J-optimal (or asymptotically J-optimal) family then a family of $(\Lambda_1\times \Lambda_2\times\cdots \times \Lambda_{n-1}\times T, w, \la)$ $ n $-D OOCs exists which is (at least) asymptotically optimal.
\end{enumerate}
\end{cor}

\begin{proof}
Follows from Theorem \ref{thm: nD-nD}, (taking $ n=1 $), Lemma \ref{lem: bound relationship}, and the bounds in Theorem \ref{thm: ddjb2}.
\end{proof}

In  \cite{MR1022081}, by considering orbits of lines in finite projective spaces, it is shown that for any prime power $ q , $ an infinite family of J-optimal $ (\frac{q^{k+1-1}}{q-1}, q+1,1) $ OOCs exits. From Corollary \label{cor: 1D optimal to nD optimal} we now see that for any factorisation $ \frac{q^{k+1-1}}{q-1}=\Lambda_1\cdot \Lambda_2\cdots \Lambda_{n-1}\cdot T $, an optimal $ (\Lambda_1\times \Lambda_2\times \cdots \times\Lambda_{n-1}\times T, q+1,1) $ $ n $-D OOC exists.

For dimensions $ n>1 $, we may also construct new optimal codes from others. 

\begin{cor}\label{cor: expand nD codes}  
Let $C$ be an $(\Lambda_1\times\Lambda_2\times \cdots \times\Lambda_{n-1} \times T,w,\la)$ $ n $-D OOC with $T=T_1\cdot T_2$.
\begin{enumerate}
\item   If $C$ is J-optimal and $f(\Lambda_1\times\Lambda_2\times \cdots \times\Lambda_{n-1} \times T,w,\la)-J(\Lambda_1\times\Lambda_2\times \cdots \times\Lambda_{n-1} \times T,w,\la)<\frac{1}{T_1}$ (in particular, if $f(\Lambda_1\times\Lambda_2\times \cdots \times\Lambda_{n-1} \times T,w,\la)$ is integral), then a J-optimal $(T_1\Lambda_1\times\Lambda_2\times \cdots \times\Lambda_{n-1} \times T,w,\la, w, \la)$ $ n $-D OOC exists.
\item   If $C$ is a member of a J-optimal family, or an asymptotically J-optimal family then a family of $(\Lambda_1\times\Lambda_2\times \cdots \times\Lambda_{n-1} \times T,w,\la)$ $ n $-D OOCs exists which is (at least) asymptotically optimal.
\end{enumerate}
\end{cor}

\begin{proof}
Follows from Theorem \ref{thm: nD-nD}, Corollary \ref{cor: bound relationship_n>1}, and the bounds in Theorem \ref{thm: ddjb2}.
\end{proof}

\section{New optimal and asymptotically optimal codes} \label{sec: New Constructions}

\subsection{Preliminaries}
Our techniques will rely heavily on the properties of finite projective and affine spaces. Such techniques have been used successfuly in the construction of infinite families of optimal OOCs of one dimension, \cite{MR1022081,MR2359316,MR2067605,MR2354005,AldMellHyperovals}, two dimensions \cite{Alderson20111187,MR2553388}, and three dimensions \cite{AldersonWSEAS2017}, \cite{Alderson2017}. We start with a brief
overview of the necessary concepts. By $PG(k,q)$ we denote the classical (or Desarguesian) finite projective geometry of dimension $k$ and order $q$. $ PG(k,q) $   may be modeled with the affine (vector) space  $AG(k+1,q)$ of
dimension $k+1$ over the finite field $GF(q)$. Under this model, points of $PG(k,q)$ correspond to 1-dimensional subspaces
of $AG(k,q)$, projective lines correspond to 2-dimensional affine subspaces, and so on.  A \textit{$d$-flat} $\Pi$ in $PG(k,q)$ is a
subspace isomorphic to $PG(d,q)$; if $d=k-1$, the subspace $\Pi$ is called a \textit{hyperplane}.
Elementary counting  shows that the number of $d$-flats in $PG(k,q)$ is given by the Gaussian coefficient
\begin{equation} \label{eqn: number of projective d-flats} \left[
\begin{array}{c}
k+1 \\
d+1 \\
\end{array}
\right]_q = \frac{(q^{k+1}-1)(q^{k+1}-q)\cdots(q^{k+1}-q^{d})}{(q^{d+1}-1)(q^{d+1}-q)\cdots(q^{d+1}-q^{d})}.
\end{equation}

In particular, the number of points of $PG(k,q)$ \textbf{is given by
$\theta(k,q)=\frac{q^{k+1}-1}{q-1}$.  We will use $\theta(k)$  to represent} this number when $ q $ is understood to be the order of the field. Further, we shall denote by $ \LL(k) $ the number of lines in $ PG(k,q) $
.  \textbf{For a point set $A$ in $PG(k,q)$ we shall denote by $\langle A \rangle$ the span of
$A$, so $\langle A \rangle=PG(t,q)$ for some $t\leq k$.}

A \emph{Singer group} of $PG(k,q)$ is a cyclic group of automorphisms acting sharply transitively on the points.  The generator of such a group is known as a \emph{Singer cycle}. Singer groups are known to exist in classical projective spaces of any order and dimension and their existence follows from that of primitive elements in a finite field.

Here, we make use of a  Singer group that is most easily understood by modelling a finite projective space using a finite field. If we let $\beta$ be a primitive element
of $GF(q^{k+1})$, the points of $\Sigma=PG(k,q)$ can be represented by the field elements
$\beta^0=1,\beta,\beta^2,\ldots,\beta^{n-1}$, where $n=\theta(k)$. 

The non-zero elements of $GF(q^{k+1})$ form a cyclic group under multiplication.
Multiplication by $\beta$ induces an automorphism, or
collineation, on the associated projective space $PG(k,q)$ (see e.g. \cite{MR0249317}). Denote by $\phi$ the collineation of
$\Sigma$ defined by $\beta^i \mapsto \beta^{i+1}$. The map $\phi$ clearly acts sharply transitively on the points of $\Sigma$.

As observed in \cite{MR2553388}, we can construct 2-dimensional codewords by considering orbits under some subgroup of $G$.
Let $n=\theta(k)=\Lambda\cdot T$ where $G$ is the Singer group of $\Sigma=PG(k,q)$. Since $G$ is cyclic there exists an unique subgroup $H$ of order $T$ ($H$ is the subgroup with generator $\phi^\Lambda$).

\begin{defn} \label{defn: Incidence matrices}
	Let $\Lambda, T$ be integers such that $n=\theta(k)=\Lambda\cdot T$.  For an arbitrary pointset $S$ in $\Sigma=PG(k,q)$ we define the $\Lambda\times T$ incidence matrix  $A=(a_{i,j})$, $0\le i\le \Lambda-1$, $ 0\le j \le T-1 $ where $a_{i,j}=1$ if and only if the point corresponding to $\beta^{i+\Lambda j}$ is in $S$.
\end{defn}

If $S$ is a pointset of $\Sigma$  with corresponding $\Lambda\times T$ incidence matrix $W$ of weight $w$, then $\phi^\Lambda$ induces a cyclic
shift on the columns of $W$.  For any such set $S$, consider its orbit $Orb_H(S)$ under the group
$H$ generated by $\phi^\Lambda$.  The set $S$ has \textit{full $H$-orbit} if $|Orb_H(S)|=T=\frac{n}{\Lambda}
$ and \textit{short $H$-orbit} otherwise. If $S$ has full $H$-orbit then a representative member of the orbit and corresponding 2-dimensional codeword is chosen. The collection of all such codewords gives rise to a $(\Lambda\times T,w,\la_a,\la_c)$ 2-D OOC, where $\la_a$ is determined by

\[
\max_{1\leq i<j \leq \;T} \left\{ |\phi^{\Lambda\cdot i}(S)\cap\phi^{\Lambda\cdot j}(S)| \right\}
\]
and $\la_c$ is determined by
\[
\max_{1\leq i,j \leq \;T} \left\{ |\phi^{\Lambda \cdot i}(S) \cap \phi^{\Lambda\cdot j} (S')| \right\}.
\]

\subsection{Construction}

Let $\Sigma=PG(k,q)$ where $G=\langle \phi \rangle $ is the Singer group of $\Sigma$ as in the previous section.  Our work will rely on the following results about orbits of flats.

\begin{thm}[Rao \cite{MR0249317}, Drudge\cite{MR1912797} ]\label{Rao}
	In $\Sigma=PG(k,q)$, there exists a short $G$-orbit of $d$-flats if and only if $gcd(k+1,d+1)\ne 1$.
	In the case that $ d+1 $ divides $ k+1 $ there is a short orbit $\SSS$ which  partitions the points of $\Sigma$ (i.e. constitutes a $d$-spread of $\Sigma$).  There is precisely one such orbit, and the $G$-stabilizer of any $\Pi\in \SSS$ is $Stab_G(\Pi)=\langle \phi^{\frac{\theta(k)}{\theta(d)}}\rangle$.\\
\end{thm}

\subsubsection{Construction 1}
For our first construction we mimic the methods of \cite{Alderson2017}, whereby codewords correspond to lines that are not contained in any element of a $ d $-spread of $ \Sigma $.\\
For $ d\ge 1 $, let $ k>1 $ such that $ d+1 $ divides $ k+1 $.  Let  $G=\langle \phi \rangle$  be the Singer group of $ \Sigma=PG(k,q) $, as detailed above, and let $\SSS$ be the  $ d $-spread determined (as in Theorem \ref{Rao}) by $G$,  where say $Stab_G(\SSS)=H = \left \langle \phi^\Lambda \right \rangle$, where $ \Lambda=\frac{\theta(k)}{\theta(d)} $.\\
Let $ \ell $ be a line not contained in any spread element (a $ d $-flat in $ \SSS $), and let $ A $ be the $ \Lambda \times \theta(d)$ projective incidence array corresponding to $ \ell $. Observe that $ \ell $ has a full $ H $-orbit.  $ H $ acts sharply transitively on the points of each spread element. It follows  that $ A $,  considered as a $ \Lambda\times  \theta(d)$ codeword, satisfies $ \la_a = 0 $. For each such line $ \ell $, we choose a representative element of it's $ H $-orbit, and include its corresponding incidence array as a codeword. The aggregate of these codewords gives an ideal $ (\Lambda\times \theta(d), q+1, 0, 1) $-3D OOC, $  C $.    Elementary counting gives
\begin{align} 
|C| & = \frac{\LL(k)-\LL(d)\cdot \frac{\theta(k)}{\theta(d)}}{\theta(d)} \nonumber  \\  
& = \frac{\theta(k)\theta(k-1)}{\theta(d)(q+1)} - \frac{\theta(d-1)\theta(k)}{\theta(d)(q+1)} \nonumber \\
& = \frac{\theta(k)}{\theta(d)(q+1)}\left[\theta(k-1)-\theta(d-1)\right]. \label{eqn: size d-spread proj}
\end{align}
Comparing (\ref{eqn: size d-spread proj})  with the bound in  Theorem  \ref{thm:ideal3djb} shows these codes to be optimal.

\begin{thm}\label{thm:optimal ideal d-spread projective}
For $ d+1$ a proper divisor of $ k+1 $, there exists a J-optimal $ (\frac{\theta(k)}{\theta(d)} \times \theta(d), q+1, 0, 1) $ 2-D OOC.
\end{thm} 

With the observation that  $ \frac{\theta{(k)}}{\theta(d)} =\theta(m-1,q^{d+1}) $, we  have shown the following.

\begin{cor}
For $ d\ge 1 $, $ m>1 $, and for any positive integral factorisation $ \Lambda_1\cdot \Lambda_2\cdots \Lambda_{n-1} =\theta(m-1,q^{d+1}) $, there exists a J-optimal (ideal) $ (\Lambda_1\times\Lambda_2\times \cdots \times \Lambda_{n-1} \times\theta(d), q+1, 0, 1) $ $ n-$D OOC .
\end{cor}

The following table will perhaps place this construction in context. Each of the optimal $ \Lambda\times T $ constructions described in the table gives rise to optimal higher dimensional OOCs, with dimensions limited by the number of distinct factors in $ \Lambda $.

\begin{table}[H]
\label{table:2}
\centering
\caption{J-optimal families of ideal 2-D OOCs that give rise to higher dimensional optimal codes.}
( $p$ prime, $q $ a prime power)
\[
\begin{array}{|l|l|c|}
\hline
 \textrm{ Parameters } & \textrm{ Conditions }   & \textrm{ Reference}\\
 \hline \hline
 \multicolumn{3}{c}{ \textrm{\textbf{Codes with} } \mathbf{\la=1}} \\
 \hline
 (\Lambda\times p,\Lambda ,0,1) &  \Lambda \le p,    & \cite{1315909} \\
  \hline
  \left( \theta(k,q^2)\times (q+1),q+1,0,1\right) &  k \ge 1  &  \cite{Alderson20111187}  \\
  \hline
  \left(\theta(k,q)\times (q-1),q,0,1\right) & k \ge 1  &  \cite{Alderson20111187}  \\
  \hline
  \left((2^{n}+1)\times \theta(k,2),2^n,0,1\right) & k \ge 1, n\ge 2  & \cite{Alderson20111187}  \\
  \hline
  (\frac{\theta(k)}{\theta(d)} \times \theta(d), q+1, 0, 1),  d<k, d+1|k+1,   & \textrm{Ideal} 
&  \textrm{Theorem } \ref{thm:optimal ideal d-spread projective}     \\[1em] 
  \hline
  \multicolumn{3}{c}{ \textrm{\textbf{Codes with} } \mathbf{\la \ge 2}} \\
 \hline
 (\Lambda\times p,\Lambda ,0,\la_c) &  \Lambda \le p, \la_c\ge 1    & \cite{1523307}\\
  \hline
  \left((q^{n}+1)\times \theta(k,q),q^n,0,q-1 \right) &  k \ge 1, n\ge 2  & \cite{Alderson20111187}  \\

  \hline
  \hline
\end{array}
\]
 \end{table}

\subsubsection{Construction 2}

In our second construction, codewords correspond to conics, and lines  in $ \Sigma=PG(3,q) $. An \textit{$m$-arc}  in $PG(2,q)$ is a collection of $m>2$ points such that no $3$ points are incident with a common line. In $PG(2,q)$, a (non-degenerate) conic is a $(q+1)$-arc. Elementary counting shows that this arc is complete (of maximal size) when $q$ is odd.  The $(q+2)$-arcs (hyperovals) exist in $PG(2,q)$ if $q$ is even and they are necessarily complete.  Conics are a special case of the so called normal rational curves. We will be interested in the existence of large collections of arcs pairwise intersecting in at most two points. From Theorem 8 of \cite{MR2359316}, and its proof, we obtain the following.  

\begin{thm}[\cite{MR2359316}] \label{thm:2-family}
In $ \Pi=PG(2,q) $ there exists a family $ \mathcal{F} $ of conics, pairwise intersecting in at most $ 2 $ points, where $ |\mathcal{F}|=q^3-q^2 $. Moreover, there is a distinguished line $ \ell $ in $ \Pi $ disjoint from each member of $ \mathcal{F} $.  
\end{thm}

Let  $G=\langle \phi \rangle$  be the Singer group as above, and let $\SSS$ be the  $ 1 $-spread determined (as in Theorem \ref{Rao}) by $G$  where say $Stab_G(\SSS)=H = \left \langle \phi^\Lambda \right \rangle$ where $ \Lambda=\frac{\theta(3)}{q+1}=q^2+1 $.\\
Through each line $ \ell $ of $ \SSS $, choose a plane $ \pi(\ell) $. As the members of $ \SSS $ are disjoint, each such plane contains precisely one member of $ \SSS $ (and therefore meets $ q^2 $ further members of $ \SSS $ in precisely one point). As $ H $ acts sharply transitively on the points of each line in $ \SSS $, each such plane has full $ H $ orbit. A dimension argument shows that any two elements in the $ H $-orbit of $ \pi(\ell) $ meet precisely in $ \ell $.  In each $ \pi(\ell) $, let $ \F(\ell) $ be a family of conics as in Theorem \ref{thm:2-family}.  Denote by $ \F =\cup \F(\ell) $, where the union is taken over all spread lines.    

Let $ \C \in \F$ be a conic, and let  $ A $ be the $ (q^2+1) \times (q+1)$ incidence array corresponding to $ \ell $. From the above, it follows  that $ A $,  considered as a codeword, satisfies $ \la_a = 0 $. For each such conic, choose a representative element of it's $ H $-orbit, and include its corresponding incidence array as a codeword. The aggregate of these codewords gives an ideal $ ( q^2+1 \times q+1, q+1, 0, 2) $-2D OOC, $ C_1 $. Note that $ \la_c=2 $ follows from the fact that two  conics  in $ \F $ are either coplanar, and therefore meet in at most two points, or are not coplanar, in which case their intersection lies on the line common to the two planes.

Note that as in Construction 1, the $ H $-orbits of non-spread lines of $ \Sigma $ correspond to an ideal  $ ( q^2+1 \times q+1, q+1, 0, 1) $-2D OOC, $ C_2 $. Since a line and a conic meet in at most two points, we have $ C= C_1\cup C_2 $ is an ideal $ ( q^2+1 \times q+1, q+1, 0, 2) $-2D OOC.  Moreover

\begin{equation}\label{eqn:size of construction 2}
|C|=  (q^2+1)\cdot(q^3-q^2)+\frac{\LL(3)-(q^2+1)}{q+1}=q(q^2+1)(q^2-q+1)
\end{equation}                
 
Comparing \ref{eqn:size of construction 2} to the bound in Theorem \ref{thm:ideal3djb} shows $ C $ to be asymptotically optimal. 

\begin{thm}\label{thm:asymp optimal ideal conics}
For $ q $ a prime power, there exists an asymptotically optimal $ (q^2+1 \times q+1, q+1, 0, 1) $ 2-D OOC.	 
\end{thm} 

\begin{cor}
For any positive integral factorisation $ \Lambda_1\cdot \Lambda_2\cdots \Lambda_{n-1} = q^{2}+1 $, there exists an asymptotically  optimal (ideal) $ (\Lambda_1\times\Lambda_2\times \cdots \times \Lambda_{n-1} \times q+1, q+1, 0, 1) $ $ n-$D OOC .
\end{cor}

\section{Conclusion}

Here, we have generalized to higher dimensions the notions of optical orthogonal codes. We establish bounds on general $ n $-dimensional OOCs, as well as specific types of ideal codes.   The bounds presented here subsume many of the existing bounds appearing in the literature that are typically applied  to codes of dimension three or less.  We present two new constructions of ideal codes; one furnishes an infinite family of optimal codes for each dimension $ n\ge 2 $, and another which provides an asymptotically optimal family for each dimension $ n\ge 2 $.



  \bibliographystyle{plain} 

\bibliography{C:/Users/UNBSJ/OneDrive/texed/my-texmf/bibtex/bib/main}

%
%
%
%

\end{document}